\RequirePackage[l2tabu,orthodox,abort]{nag}

\documentclass{birkjour}[draft]

\usepackage{fixltx2e}

\usepackage{amssymb,amsfonts,amsmath,
amsrefs,enumerate,color,url,
mathbbol,dsfont}
\usepackage{mathrsfs}
\usepackage{tikz}

\usepackage[T1]{fontenc}
\usepackage[cp1250]{inputenc}

\usepackage[normalem]{ulem}

\usepackage[strict=true]{csquotes}

\usepackage{graphicx,tikz}

 \newtheorem{thm}{Theorem}[section]
 \newtheorem{cor}[thm]{Corollary}
 \newtheorem{lem}[thm]{Lemma}
 \newtheorem{prop}[thm]{Proposition}
 \theoremstyle{definition}
 
 \theoremstyle{remark}
 \newtheorem{rem}{Remark}
 
 \numberwithin{equation}{section}
 
\newcommand{\ason}{A_{n\alpha,n\beta}^{\textnormal{s-o}}}
\newcommand{\askew}{A^{\textnormal {{skew}}}_{\alpha,\beta}}

\newcommand{\xreg}{\mathsf F_{\textnormal {reg}}}
\newcommand{\eso}{E_{\alpha,\beta}^{\textnormal{s-o}}}
\newcommand{\eson}{E_{n\alpha,n\beta}^{\textnormal{s-o}}}
\newcommand{\eskew}{E_{\alpha,\beta}^{\textnormal{skew}}}
\newcommand{\esop}{
E_{\alpha,\beta}^{\textnormal{o-s}}}
\newcommand{\esopn}{E_{n\alpha,n\beta}^{\textnormal{o-s}}}

 \DeclareMathOperator{\cosso}{Cos_{\alpha,\beta}^{\textnormal{s-o}}}
 \DeclareMathOperator{\cossot}{Cos_{\alpha,\beta}^{\textnormal{o-s}}}
  \DeclareMathOperator{\cossotn}{Cos_{\textrm{$n$}\alpha,\textrm{$n$}\beta}^{\textnormal{o-s}}}
 \DeclareMathOperator{\cosson}{Cos_{\textrm{$n$}\alpha,\textrm{$n$}\beta}^{\textnormal{s-o}}}
 
\DeclareMathOperator{\cskew}{Cos_{\alpha,\beta}^{\textnormal{skew}}}
\DeclareMathOperator{\cweks}{Cos_{\alpha,\beta}^{\textnormal{weks}}}

\DeclareMathOperator{\cskewr}{Cos_{\beta,\alpha}^{\textnormal{skew}}}
\usepackage{mathtools}

\usepackage{enumitem}
\renewcommand{\theenumi}{\textup{(\alph{enumi})}}
\renewcommand{\labelenumi}{\theenumi}

\usepackage{hyperref}
\begin{document}

\title[Skew and snapping-out Brownian motions]{Approximation of skew Brownian motion by snapping-out Brownian motions}

\author[A. Bobrowski]{Adam Bobrowski}

\address{
Lublin University of Technology\\
Nadbystrzycka 38A\\
20-618 Lublin, Poland}
\email{a.bobrowski@pollub.pl}

\author[E. Ratajczyk]{El\.zbieta Ratajczyk}
\address{
Lublin University of Technology\\
Nadbystrzycka 38A\\
20-618 Lublin, Poland}

\email{e.ratajczyk@pollub.pl}

\newcommand{\cxi}{(\xi_i)_{i\in \N} }
\newcommand{\lam}{\lambda}
\newcommand{\eps}{\varepsilon}
\newcommand{\ud}{\, \mathrm{d}}
\newcommand{\mud}{\mathrm{d}}
\newcommand{\pr}{\mathbb{P}}
\newcommand{\R}{\mathbb{R}}
\newcommand{\e}{\mathrm {e}}
\newcommand{\tif}{\widetilde {f}}
\newcommand{\Id}{{\mathrm{Id}}}
\newcommand{\cic}{C_{\mathrm{mp}}}
\newcommand{\pol}{{\textstyle \frac 12}}
\newcommand{\ft}{{\textstyle \frac 1t}}
\newcommand{\es}{\textnormal T}

\newcommand{\cee}{\mathfrak  C[-\infty,\infty]}
\newcommand{\cod}{\mathfrak C_{\mathrm{odd}}[-\infty,\infty]}
\newcommand{\cev}{\mathfrak C_{\mathrm{even}}[-\infty,\infty]}
\newcommand{\cevr}{C_{\mathrm{even}}(\mathbb{R})}
\newcommand{\codr}{C_{\mathrm{odd}}(\mathbb{R})}
\newcommand{\cez}{C_0(0,1]}
\newcommand{\fod}{f_{\mathrm{odd}}} 
\newcommand{\fev}{f_{\mathrm{even}}} 
\newcommand{\sem}[1]{\mbox{$\left (\e^{t{#1}}\right )_{t \ge 0}$}}
\newcommand{\semi}[1]{\mbox{$\left ({#1}\right )_{t > 0}$}}
\newcommand{\semt}[2]{\mbox{$\left (\e^{t{#1}} \otimes_\varepsilon \e^{t{#2}} \right )_{t \ge 0}$}}
\newcommand{\tr}{\textcolor{red}}
\newcommand{\wt}{\widetilde}

\newcommand{\tcm}{\textcolor{magenta}}
\newcommand{\ecm}{\textcolor{olive}}
\newcommand{\tcb}{\textcolor{blue}}
\newcommand{\dx}{\ \textrm {d} x}
\newcommand{\dy}{\ \textrm {d} y}
\newcommand{\dz}{\ \textrm {d} z}
\newcommand{\di}{\textrm{d}}
\newcommand{\tcg}{\textcolor{green}}
\newcommand{\lc}{\mathfrak L_c}
\newcommand{\ls}{\mathfrak L_s}
\newcommand{\grat}{\lim_{t\to \infty}}
\newcommand{\grar}{\lim_{r\to 1-}}
\newcommand{\graR}{\lim_{R\to 1+}}
\newcommand{\grak}{\lim_{\kappa \to \infty}}
\newcommand{\gra}{\lim_{x\to \infty}}
\newcommand{\grae}{\lim_{\eps \to 0}}
\newcommand{\gran}{\lim_{n\to \infty}}
\newcommand{\rez}[1]{\big (\lam - #1\big)^{-1}}
\newcommand{\papa}{\hfill $\square$}
\newcommand{\papap}{\end{proof}}
\newcommand {\x}{\cerf}
\newcommand{\aex}{A_{\mathrm ex}}
\newcommand{\jcg}[1]{\big ( #1 \big )_{n\ge 1} }
\newcommand{\injtp}{\x \hat \otimes_{\varepsilon} \y}
\newcommand{\pin}{\|_{\varepsilon}}
\newcommand{\mc}{\mathcal}
\newcommand{\inter}{\left [0, 1\right ]}
\newcommand{\ha}{\mathfrak {H}}
\newcommand{\dom}[1]{D(#1)}
\newcommand{\mquad}[1]{\quad\text{#1}\quad}
\newcommand{\lil}{\lim_{\lam \to \infty}}
\newcommand{\lilz}{\lim_{\lam \to 0}}

\newcommand{\ce}{\mathcal C}
\newcommand{\cerl}{\mathfrak C[-\infty,0]}

\newcommand{\cerp}{\mathfrak C[0,\infty]}
\newcommand{\cer}{\mathfrak  C[-\infty,\infty]}
\newcommand{\cerf}{\mathfrak  C(\R_\sharp)}
\newcommand {\xo}{\mathfrak  C_{\textnormal{ov}}(\R_\sharp)}
\newcommand{\cerr}{\mathfrak C_R^\alpha}
\newcommand{\cef}{\mathfrak C_F^\alpha}
\newcommand{\comega}{C_\omega[-\infty,\infty]}
\newcommand{\fo}{f_{\textrm{o}}}
\newcommand{\fe}{f_{\textrm{e}}}
\newcommand{\wh}{\widehat}

\newcommand{\sq}{\gamma}
\newcommand{\su}{{\alpha+\beta}}
\newcommand{\al}{\alpha}
\newcommand{\be}{\beta}

\newcommand{\yy}{\mathcal{C}_{\al,\be}^{\textnormal{skew}}}
\newcommand{\zzz}{\mathcal{C}_{\al,\be}^{\textnormal{weks}}}
\newcommand{\y}{\mathcal{C}_{\al,\be}^{\textnormal{s-o}}}
\newcommand{\yskew}{\mathcal{C}_{\al,\be}^{\textnormal{skew}}}
\newcommand{\z}{\mathcal{C}_{\al,\be}^{\textnormal{o-s}}}
\renewcommand{\q}{P_{\al,\be}^{\textnormal{o-s}}}
\newcommand{\qn}{P_{n\al,n\be}^{\textnormal{o-s}}}
\newcommand{\qweks}{P_{\al,\be}^{\textnormal{weks}}}
\newcommand{\p}{P_{\al,\be}^{\textnormal{s-o}}}
\newcommand{\pn}{P_{n\al,n\be}^{\textnormal{s-o}}}
\newcommand{\pskew}{P_{\al,\be}^{\textnormal{skew}}}

\newcommand{\fl}{f_{\ell}}
\newcommand{\fr}{f_{\textnormal r}}
\newcommand{\fle}{f_1}
\newcommand{\fre}{f_2}
\newcommand{\wfl}{\widetilde{\fl}}
\newcommand{\wfr}{\widetilde{\fr}}
\newcommand{\gl}{g_1}
\newcommand{\gr}{g_2}
\newcommand{\gle}{g_1}
\newcommand{\gre}{g_2}
\newcommand{\hl}{h_{\textnormal l}}
\newcommand{\hr}{h_{\textnormal r}}
\newcommand{\apb}{{\alpha+\beta}}
\newcommand{\Ff}{\mathfrak{F}}
\newcommand{\Ffe}{k}
\newcommand{\F}{\mathcal{F}}
\newcommand{\mab}{M_{\alpha,\beta}}

\newcommand{\rla}{R_\lam}

\newcommand{\fh}{h} 
\newcommand{\ph}{\phi} 

\newcommand{\ced}{C_{\textnormal{D}}}

\makeatletter
\newcommand{\normt}{\@ifstar\@normts\@normt}
\newcommand{\@normts}[1]{%
  \left|\mkern-1.5mu\left|\mkern-1.5mu\left|
   #1
  \right|\mkern-1.5mu\right|\mkern-1.5mu\right|
}
\newcommand{\@normt}[2][]{%
  \mathopen{#1|\mkern-1.5mu#1|\mkern-1.5mu#1|}
  #2
  \mathclose{#1|\mkern-1.5mu#1|\mkern-1.5mu#1|}
}
\makeatother

\thanks{Version of \today}

\begin{abstract} We elaborate on the theorem saying that as permeability coefficients of snapping-out Brownian motions tend to infinity in such a way that their ratio remains constant, these processes converge to a skew Brownian motion. In particular, convergence of the related semigroups, cosine families and projections is discussed. 
\end{abstract}

\maketitle

\section{Introduction}

\subsection{Skew Brownian motion}
 As described in \cite{lejayskew} `The skew Brownian Motion appeared in the '70 in \cite{ito,walsh} as a natural generalization of the Brownian motion: it is a process that behaves like a Brownian motion except that the sign of each excursion is chosen using an independent Bernoulli random variable'. This characteristic property of skew Brownian motion is expressed in the boundary conditions 
\begin{equation}\label{intro:1} f''(0+)=f''(0-) \mquad{ and } \beta f'(0-)=\alpha f'(0+),\end{equation}
where  $\alpha $ and $\beta$ are  non-negative parameters such that $\alpha+\beta>0$.
More precisely, the skew Brownian motion is an honest Feller process on $\R$, and as such can be described by means of 
its generator, that is, a Laplace operator, say, $\askew$, in the space $\cer$ ($[-\infty,\infty]$ is a two-point compactification of $\R$) of continuous functions on $\R$ that have finite limits at $\infty$ and $-\infty$. The domain of $\askew$ is composed of functions satisfying the following three conditions: 
\begin{itemize}
\item[(a)] $f$ is twice continuously differentiable in both
  $(-\infty,0]$ and $[0,\infty)$, separately, with left-hand and
  right-hand derivatives at $x=0$, respectively, 
\item[(b)] both the limits $\lim_{x\to \infty} f'' (x)$ and
  $\lim_{x\to -\infty} f''(x) $ exist and are finite {(it follows that, in fact, they
  have to be equal to $0$)}, and
 \item[(c)] conditions \eqref{intro:1} hold; note that the first of them implies that, although $f'(0)$ need not exist, it is meaningful to speak of $f''(0)$.  
\end{itemize}
Furthermore, we
   define \[ \askew f \coloneqq f'' .\]
Even though, for reasons that will become obvious in what follows, it is convenient to consider all non-negative parameters $\alpha $ and $\beta$ in \eqref{intro:1} (such that $\alpha+\beta>0$), their probabilistic meaning is more clear when they are normalized by dividing both of them  by $\alpha + \beta$. Then the new $\alpha$ is the probability that the sign of excursion is chosen to be positive, and the new $\beta$ is the probability that the sign is negative.


 \newcommand{\aso}{A_{\al,\be}^{\textnormal{s-o}}}
 \newcommand{\aos}{A_{\al,\be}^{\textnormal{o-s}}}
 
 The paper by A.~Lejay cited above discusses a number of constructions of the skew Brownian motion that appeared in the literature since the process was first discovered,  and various contexts in which this process and its generalizations are studied; see also \cite{yor97}*{p. 107}  and \cite{manyor}*{pp. 115--117}. It is well-known, for example, that skew Brownian motion can be obtained in the limit procedure of Friedlin and Wentzell's averaging principle \cite{fwbook}---see \cite{fw}*{Thm. 5.1}, comp. \cite{emergence}*{Eq. (3.1)}. Much more recently, in \cite{abtk}, a link has been provided between skew Brownian motion and kinetic models of motion of a phonon involving an interface, of the type studied in
\cite{tomekkinetic,tomekkinetic3,tomekkinetic2,tomekkinetic1}, and the telegraph
process with elastic boundary at the origin \cite{wlosi,wlosi1}.

\subsection{Snapping-out Brownian motion}\label{sec:sobm}

The goal of this paper is to use the recent results of \cite{aberman} to strengthen the limit theorem of \cite{tombatty} (see also \cite{knigaz}*{Chapter 11}) saying that the skew Brownian motion can be obtained as a limit of \emph{snapping out Brownian motions}.  
   
The snapping out Brownian motion is a diffusion on two half-lines,  $(-\infty,0)$ and $(0,\infty)$, separated by a semi-permeable membrane located at~$0$. It can be described by a Feller semigroup of operators in the space \[\cerf \] of continuous functions on 
\[ \R_\sharp\coloneqq  [-\infty,0-] \cup [0+,\infty] \]
where $0-$ and $0+$ are two distinct points, representing 
positions to the immediate left and to the immediate right of the membrane; alternatively, members of $\cerf$ 
can be seen as continuous functions on $(-\infty,0)\cup(0,\infty)$ that have finite limits at $\pm \infty$ and one-sided finite limits at $0$. The membrane, in turn, is characterized by two non-negative parameters, say, $\alpha$ and $\beta$, describing its permeability for a particle diffusing from the left to the right and from the right to the left, respectively. More precisely, given such  $\al$ and $\be$, we define the generator of the snapping out Brownian motion as follows: it is the operator $\aso$ in $\x$  given by 
 \[ \aso f \coloneqq f'', \]
on the domain composed of twice continuously differentiable $f \in \x $ such that $f', f''\in \x$ and 
\begin{align}\label{intro:2}\begin{split}
f'(0-)&=\al(f(0+)-f(0-)), \\
f'(0+)&=\be(f(0+)-f(0-)). 
\end{split}\end{align}

These \emph{transmission conditions} are well-known in the literature and can be plausibly interpreted: according to Newton's Law of Cooling, the temperature at $x=0$ changes at a rate proportional to the difference of temperatures on  
either sides of the membrane, see \cite[p. 9]{crank}. In this context, J. Crank uses the term \emph{radiation boundary condition}. 
In the context of passing or diffusing through membranes, analogous transmission conditions were introduced by J.E.~Tanner \cite[eq. (7)]{tanner}, who studied diffusion of particles through a sequence of permeable barriers (see also Powles et al. \cite[eq. (1.4)]{powles}, for a continuation of the subject).  In \cite{andrews} (see e.g. eq. (4) there) similar conditions are used in describing absorption and desorption phenomena. We refer also to \cite{fireman}, where a~compartment model with permeable walls (representing e.g., cells, and axons in the white matter of the brain in particular) is analyzed, and to eq. [42] there. Among more recent literature involving relatives of \eqref{intro:2}, one should mention the model  of inhibitory synaptic receptor dynamics of P.~Bressloff
\cite{bressloff2023}; see also \cite{bressloff2022} and references given in these papers.

In \cite{bobmor} (see also \cite{knigaz}*{Chapters 4 and 11}), relations \eqref{intro:2} were put in the context of Feller--Wentzel boundary conditions, and this allowed describing the related process by means  of the \emph{L\'evy local time}. Later, in 
\cite{lejayn}, a stochastic construction was provided, and the term \emph{snapping out Brownian motion} was coined. To summarize these findings, we note that conditions \eqref{intro:2} are akin to the 
elastic barrier condition: An elastic Brownian motion on $\R^+:=[0,\infty)$ (see e.g. \cite{ito,karatzas}) is the process with generator 
$ Gf = \frac 12 f'' $
defined on the domain composed of $f\in \mathfrak{C}^2[0,\infty]$ satisfying the \emph{Robin boundary condition} (known also as \emph{elastic barrier condition}): \[f'(0+)=\beta f(0+).\]  In this process, the state-space is $\R^+$, and each particle performs a standard Brownian motion while away from the barrier $x=0.$ Each time a particle touches the barrier it is reflected, but its time spent at the boundary, the {local time} mentioned above, is measured and after an exponential time with parameter $\beta$ with respect to the local time, the particle is killed and no longer observed. The second condition in \eqref{intro:2} expresses the fact that in the snapping out Brownian motion, each particle diffusing on the right half-axis, instead of being killed after the time described, is transferred to the other side of the membrane $x=0$. In other words, the particle after some time spent `at' the right side of the boundary, permeates to the the left side. Similarly, a particle on the left side of the boundary permeates through the membrane after a random time characterized by the coefficient $\alpha$. In particular, we see that $\alpha$ and $\beta$ are permeability coefficients: the larger they are the shorter is the time for particles to diffuse through the membrane.  

\subsection{Approximation of skew B.M. by snapping out B.M.}
Since $\alpha$ and $\beta$ in \eqref{intro:2} are permeability coefficients, it may be thought reasonable that letting them to be infinite should lead to the case in which the membrane is completely permeable, that is, to the standard Brownian motion. As it turns out, however, the matter is not so simple: in the limit some kind of asymmetry in the way particles pass to the left and to the right through the membrane remains. 

To explain this phenomenon in more datail, let us replace $\alpha$ and $\beta$ in the transmission conditions \eqref{intro:2} by $n\alpha$ and $n\beta$, respectively, and let $n \to \infty$. Heuristically, it is then clear that the limit conditions should read \[ f(0+)=f(0-) \mquad{ and } \beta f'(0-)=\alpha f'(0+).\]
The first of these relations tells us that we should work with $f\in \x$ that are continuous at $x=0$, that is with $f\in \cer$, and the other is precisely the second of  conditions that characterize the skew Brownian motion, see \eqref{intro:1}. Hence, we anticipate a theorem saying that a skew Brownian motion is a limit of snapping out Brownian motions, provided that permeability coefficients $\alpha$ and $\beta$ of the semi-permeable membrane converge to infinity whereas their ratio remains constant.

As we know from \cite{tombatty}, these intuitions can be transformed into a formal theorem saying that 
for any $s>0$, 
\begin{align}\label{intro:3} \gran \sup_{t\in [0,s]}\big \|\e^{t \ason}f - \e^{t\askew}f\big\| = 0, \qquad f \in \cer, \end{align}
where  $\cer$ is seen as the subspace of $\x$ composed of functions $f$ such that $f(0-)=f(0+)$, $\{\e^{t\aso},\, t\ge 0 \}$ is the Feller semigroup describing snapping-out Brownian motion, and $ \{\e^{t\askew}, \,t\ge 0\}$ is that describing the skew Brownian motion. In view of the Trotter--Sova--Kurtz--Mackevi\u cius theorem \cite{kallenbergnew}, this result can be seen as expressing a convergence of the random processes involved. 

\subsection{The goal of the paper}
The goal of the paper is to strengthen the result of \cite{tombatty} by 
exhibiting a number of phenomena accompanying convergence \eqref{intro:3}. These can be summarized as follows. 

\begin{itemize}
\item [(A)] \emph{Convergence of cosine families.} 
As proved in \cite{tombatty}, each $\aso$ generates not only a semigroup but also a cosine family, and the latter is in a sense more fundamental than the semigroup. We will show that on $\cer$ (but not outside of this space) not only the semigroups generated by $\ason$ converge; so do also the related cosine families. Moreover, the limit is in fact uniform with respect to $t\in \R$, and for Lipschitz continuous functions the rate of convergence can be estimated---see Section \ref{cosf}, and Theorem \ref{thm:skew} (b) and Remarks \ref{rem:one} and \ref{rem:rate} in particular.

\item [(B)]\emph {Convergence of semigroups outside of $\cer$.} As a consequence of (A), for $f\not \in \cer $ the limit $\gran \e^{t\ason}f $ still exists for all $t>0$, but is merely uniform with respect to $t$ in compact subintervals of $(0,\infty)$---see Theorem \ref{thm:skew} (c).

\item [(C)] \emph{Asymptotic behavior.} We will show that the cosine families generated by $\aso$ have the same time averages as those generated by $\askew$---see Section \ref{sec:ab}.

\item [(D)]  \emph {Convergence of projections.} It is also proved in \cite{aberman} that each transmission condition \eqref{intro:2} shapes unequivocally an extension of each member of $\cerf$ to a pair of members of $\cer$, that is, to a member of the Cartesian product
\[ \ce \coloneqq (\cer)^2. \]
The subspace $\y\subset \ce $ of such extensions, besides being invariant under the Cartesian product basic cosine family in $\ce$ (see \eqref{intro:7}), is \emph{complemented}. Therefore, there are related projections $\p \in \mc L(\ce)$ on $\y$. We prove that 
 $\gran \pn$ exists and is a projection on the subspace of extensions unequivocally shaped by the skew Brownian motion boundary condition \eqref{intro:1}. See Section \ref{sec:cop}.

\item [(E)] \emph{Convergence of complementary cosine families and semigroups.}  
As also proved in \cite{aberman}, the complementary subspace of $\y$, say, $\z$, in $\ce$ is unequivocally shaped by certain transmission conditions too. As a result there are complementary cosine families and semigroups to those generated by  $\aso$. We prove that these cosine families and semigroups converge similarly as those generated by $\aso$. 
Surprisingly, the limit semigroup and the limit cosine family turn out to be isomorphic copies of the semigroup and cosine family related to another skew Brownian motion, with the role of coefficients $\alpha$ and $\beta$ reversed---see Section \ref{sec:cocf}. 
\end{itemize}

Additional results completing the paper are presented in Section \ref{sec:ar}.

\subsection{Commonly used notation}\label{cun} 
\subsubsection{Constants $\alpha$ and $\beta$}
Throughout the paper we assume that $\alpha$ and $\beta$ are fixed non-negative constants. Also, to shorten formulae, we write 
\begin{align*} 
\sq\coloneqq \sqrt{2(\al^2+\be^2)}.
\end{align*}
Since the case of $\alpha=\beta=0$ is not interesting, in what follows we assume that ${\alpha+\beta}>0$, implying that $\gamma>0$ as well. 
\subsubsection{Limits at infinities, and transformations of functions} For $f\in \x$, we write 
\[f(\pm \infty)\coloneqq \lim_{x\to \pm \infty} f(x).\]
Moreover, we define
$ f^\es, f^e $ and $f^o$, also belonging to $\x$, by 
\begin{align} 
f^\es (x)& \coloneqq  f(-x),   \qquad  x \in \R{\setminus \{0\}}, \nonumber  \end{align} 
and \begin{align} f^e  &\coloneqq \pol (f + f^\es), \qquad  f^o  \coloneqq \pol (f - f^\es ). \label{elek}
\end{align}




\subsubsection{Convolution with exponential function} For $a>0$ we introduce
\[ e_a(x)\coloneqq \e^{-ax},\qquad x\in \R^+, \]
 and for $f\in \x$, we write \[ e_a *f  (x) \coloneqq \int_0^x \e^{-a(x-y)} f(y) \ud y , \qquad x \in \R^+.\]

\subsubsection{The restriction operator}\label{tro}
The following restriction operator, denoted $R$ and mapping  $\ce$ onto the space $\x$ of Section \ref{sec:sobm}, will be of key importance in the entire paper. By definition $R$ assigns to a  pair $(f_1,f_2)\in \ce$ the member $f$ of $\x$ given by 
\[ f(x) = f_1(x), \text{ for } x<0  \mquad { and } f(x) = f_2(x), \text{ for } x >0. \]

\section{Convergence of cosine families and semigroups describing snapping-out Brownian motions}\label{cosf} 

Before our main theorem in this section is presented we need to explain the nature of the result it generalizes, i.e., to put \eqref{intro:3} in the context of general theory of convergence of semigroups.  Also, in Section \ref{sec:cf} we recall the notion of cosine family.

\subsection{Theory of convergence of semigroups of operators}

The main idea of the Trotter--Kato--Neveu convergence theorem \cite{abhn,engel,goldstein,pazy}, a cornerstone of the theory of convergence of semigroups \cite{knigaz,bobrud}, is that convergence of resolvents of \emph{equibounded} semigroups in a Banach space $\mathsf F$, gives an insight into convergence of the semigroups themselves. Hence, in studying the limit of semigroups, say, $\{\e^{tB_n}, t \ge 0\}, n\ge 1$, generated by the operators $B_n, n \ge 1$ we should first establish the existence of 
the strong limit 
\[ \rla \coloneqq \gran \rez{B_n}. \]
The general theory of convergence (see \cite{knigaz}*{Chapter 8}) covers also the case in which, unlike in the  classical version of the Trotter--Kato--Neveu theorem, the (common) range of the operators $\rla, \lam >0$  so-obtained is \emph{not} dense in  $\mathsf F$, and stresses the role of  the so-called \emph{regularity space}, defined as the closure of the range of $\rla$: 
\begin{equation}\label{skew:reg} \xreg \coloneqq cl (Range \rla) \subset \mathsf F. \end{equation}
Namely, $\xreg$ turns out to be composed of $f\in \mathsf F$ such the limit $T(t)f \coloneqq \gran \e^{tB_n}f $ exists and is uniform with respect to $t$ in compact subintervals of $[0,\infty)$; then $\{T(t),\, t \ge 0\}$ is a strongly continuous semigroup in $\xreg$.

Condition \eqref{intro:3} (more precisely, the theorem proved in \cite{tombatty}) is thus a  typical result of convergence theory: it characterizes the regularity space for $B_n\coloneqq \ason$ as equal to $\cer \subset \x$, and identifies $\{T(t),\, t\ge 0\}$ as $\{\e^{\askew}, t \ge 0\}$. It should be stressed, though, that this statement does not exclude the possibility of the existence of $f\not \in \cer$ such that $\jcg{\e^{t\ason}f}$ converges for all $t>0$. Such \emph{irregular} convergence of semigroups, which is known to be always uniform with respect to $t$ in compact subsets of $(0,\infty)$ --- see \cite{note} or \cite{knigaz}*{Thm 28.4} --- is not so uncommon, especially in the context of singular perturbations \cite{banmika,banalachokniga,knigaz}, but needs to be established by different means. We will prove that the limit $\gran \e^{t\ason}f$ exists for all $t>0$ and $f\in \x$, by examining the related cosine families. 

\subsection{Cosine families}\label{sec:cf} 
A strongly continuous family $\{C(t), t \in \R\}$ of operators in a Banach space $\mathsf F$ is said to be a cosine family iff $C(0)$ is the identity operator and 
\[ 2 C(t) C(s) = C(s+t) + C(t-s), \qquad t,s\in \R.\]
The generator of such a family is defined by 
\[ Af = \lim_{t\to 0} 2t^{-2}{(C(t)f- f)}\]
for all $f \in \mathsf F$ such the limit on the right-hand side exists. 
For example, in $\cer$ there is the \emph{basic cosine family} given by  
\begin{equation}\label{intro:0} C(t) f(x) = \pol [ f(x+t) + f(x-t) ], \qquad x \in \R, t \in \R.\end{equation}
Its generator is the  one-di\-men\-sion\-al Laplace operator $f\mapsto f''$ with domain composed of twice continuously differentiable functions on $\R$ such that $f'' \in \cer$. 

Each cosine family generator, say, $A$, is automatically the generator of a strongly continuous semigroup (but not vice versa). The semigroup  such $A$ generates is given by the \emph{Weierstrass formula} (see e.g. \cite{abhn}*{p. 219})
\begin{equation}\label{cosf:1}T(0)f =f    \mquad {and} T(t) f  = {\textstyle \frac 1{2\sqrt{\pi t}}} \int_{-\infty}^\infty \e^{-\frac {s^2}{4t}} C(s) f  \ud s, \qquad t >0, f \in \mathsf F.\end{equation} 
This formula expresses the fact that the cosine family is then a more fundamental object than the semigroup, and properties of the semigroup can be hidden in those of the cosine family.

In our main theorem of this section, we will show that \eqref{intro:3} can be strengthened by proving that not only the semigroups generated by $\aso$ converge but so also do the related cosine families. The fact that each $\aso$ is a generator of a cosine family was proved in \cite{tombatty}, but from the analysis presented there it is not clear that the cosine families generated by $\aso$'s are equibounded or that they do converge (see Remark 6.1 in that paper).  
\subsection{Convergence theorem for solution families}\label{sec:tfct}

We are finally ready to present our main theorem in this section. 

\begin{thm}\label{thm:skew} \ \ 
\begin{itemize}
\item [(a)] The cosine families $\{ \cosso (t), \,t \in \R\}$ generated by $\aso$ are equibounded: 
\begin{equation} \|\cosso (t)\| \le 5, \qquad \alpha,\beta\ge 0, t \in \R.\label{skew:2}\end{equation}
\item [(b)] We have 
\begin{equation}\label{skew:3} \gran \sup_{t \in \R}\big \| \cosson (t)f - \cskew (t) f\big\| =0, \qquad  f \in \cer, \end{equation}
 where $\{\cskew (t), \, t \in \R\}$ is the cosine family generated by $\askew$.  For $f \not \in \cer$, this limit exists for no $t\not =0$. 
 \item [(c)] The limit
\begin{equation}\label{skew:4} \gran \e^{t\ason} f \end{equation}
exists for all $t>0$ and $f\in \x,$ and for $f\not \in \cer $ it is uniform with respect to $t $ in compact subintervals of $(0,\infty)$. \end{itemize} 
\end{thm}

Information on the rate of convergence will be given in Remark \ref{rem:rate} further down.

\begin{rem}\label{rem:one} Thesis (b) and the Weierstrass formula imply 
that the limit in  \eqref{intro:3} is in fact uniform in $t\ge 0$. To see this, let $f\in \cer $ and $\eps >0$ be fixed and let $n_0$ be so large that 
\[\big \| \cosson (s) f - \cskew  (s)f\big \|\le  \eps \]
for $n\ge n_0$ and $s\in \R$. Then, for any $t>0$,
\begin{align*} \big\|\e^{t\ason} f - \e^{t\askew} f\big\|&\le  {\textstyle \frac 1{2\sqrt{\pi t}}} \int_{-\infty}^\infty \e^{-\frac {s^2}{4t}}\big \|\cosson (s) f  -  \cskew (s) f\big \| \ud s \\
 &=  {\textstyle \frac 1{2\sqrt{\pi t}}} \int_{-\infty}^\infty \e^{-\frac {s^2}{4t}} \eps \ud s =\eps, 
 \end{align*} 
(and for $t=0$ the left hand side is $0$). Hence, (b) improves \eqref{intro:3}.
\end{rem}

Our proof of Theorem \ref{thm:skew} hinges on the representation of cosine families $\{\cosso (t), t\in \R\}$ and $\{\cskew (t), t \in \R\}$ by means of the \emph{basic Cartesian product cosine family} \[ \{\ced(t),\, t\in \R\} \] (`D' for `Descartes') defined in $\ce$ by the formula 
\begin{equation}\label{intro:7} \ced (t) (f_1,f_2) = (C(t)f_1,C(t)f_2), \qquad t\in \R, f_1,f_2 \in \cer ,\end{equation}
where $C(t)\in \mc L(\cer)$ are introduced in \eqref{intro:0}. The following two propositions were obtained in \cite{aberman}*{Section 2} and \cite{abtk}*{Section 6}, respectively. Extension operators featured in these propositions are illustrated in our Figure~\ref{rys3}.

\begin{figure}
 \begin{tikzpicture}[scale=0.7]
 \draw [->] (-5,0) -- (5.2,0);
\draw [->] (0,-1.1)-- (0,2);
\draw[semithick,domain=0:5, samples=60] plot (\x, {0.3+exp(-\x)*cos(deg(3*\x))});
\draw[dashed,thick, domain=-5:0, samples=60] plot (\x, {0.3+3*exp(\x)- 2*exp(2*\x)-\x*exp(\x)}); 
\draw [ semithick] plot [smooth] coordinates {(-5,-0.75) (-4,0.-0.73)  (-3,-0.2) (-2,0.35) (-0.7,-0.2) (0,-0.15)};
\draw [thick, dashed]  plot [smooth] coordinates {(0,-0.15) (0.5,-0.1) (1,-0.5)  (2,-0.8)  (3,-0.5) (4,-1)(5,-1)};
\node [above] at (2.9,-1.3) {$\wfl $};
\node [below] at (-1.4,2.1) {$ \wfr$};
\end{tikzpicture}
\caption{Two extensions of a single function $f \in \x$ (solid lines): extension $\wfl$  of its left part, and extension $\wfr$ of its right part. }\label{rys3} 
\end{figure}

 \begin{prop} \label{tfct:prop1} The cosine family $\{\cosso (t),t \in \R\}$ in $\x$ generated by  $\aso$  is given by the abstract Kelvin formula 
\begin{equation}\label{ab:6} \cosso (t) =R \ced (t) \eso, \qquad t \in \R. \end{equation}
Here, (a) $R$ is the restriction operator of Section \ref{tro}, and (b) 
the extension operator $\eso\colon\x \to \ce $ maps an $ f \in \x$ to the pair $(\wfl,\wfr)\in \ce$ of 
 extensions of its left and right parts determined  by 
\begin{align}
 \widetilde {\fl}(x)&= \begin{cases*}f(x),& $x<  0$,  \\ 
 f(-x)+2\al  e_{\su}* [f-f^\es] (x), & $x> 0,$  \end{cases*} \label{needed}
 \end{align}
and 
\begin{align}
  \widetilde {\fr}(x)&= \begin{cases*} f(-x)- 2\be  e_\su  * [f-f^\es](-x) , & $x< 0$,\\
f(x),&  $x>0$. \end{cases*} \label{ajednak} \end{align} 
In particular, 
\begin{equation}\label{ab:5a} \|\eso\|\le 1 +\textstyle{\frac {4}{\alpha+\beta}}\max (\alpha,\beta)\le 5.\end{equation}
 \end{prop}

 \begin{prop} \label{tfct:prop2} The cosine family $\{\cskew (t),t \in \R\}$ in $\cer$ generated by  $\askew$  is given by the abstract Kelvin formula: 
\begin{equation}\label{ab:6a} \cskew (t) =R \ced (t) \eskew, \qquad t \in \R, \end{equation}
in which (a) $R$ is the restriction operator of Section \ref{tro}, and (b) 
the extension operator $\eskew\colon\cer \to \ce $ maps an $ f \in \cer$ to the pair $(\wfl,\wfr)\in \ce$ of 
 extensions of its left and right parts defined by
 \begin{align}\label{skew:4.1}
 \widetilde {\fl}(x)&= \begin{cases*}f(x),& $x\leq  0$, \\ 
\frac{\be-\al}{\apb} f(-x)+\frac{2\al}{\apb} f(x), &  $x> 0,$  \end{cases*}
 \end{align}
and
\begin{align}\label{skew:4.2}
  \widetilde {\fr}(x)&= \begin{cases*} \frac{2\be}{\apb} f(x)+\frac{\al-\be}{\apb} f(-x) , & $x< 0$,\\
f(x),&  $x\geq  0$.  \end{cases*} \end{align} 
 \end{prop}

\begin{rem} The cartesian product basic cosine family leaves the image of $\cer$ via $\eskew$ invariant. On the other hand, \eqref{skew:4.1} and \eqref{skew:4.2} make it clear that  $ \widetilde {f_\ell}(0)= f(0)= \widetilde {f_{\textnormal r}}(0)$. As a result $ R \ced (t) \eskew f$ is a member of $\cer$ for $f\in \cer$. 

\end{rem}

\begin{figure} [tb]\centering
\includegraphics[width=0.4\textwidth]{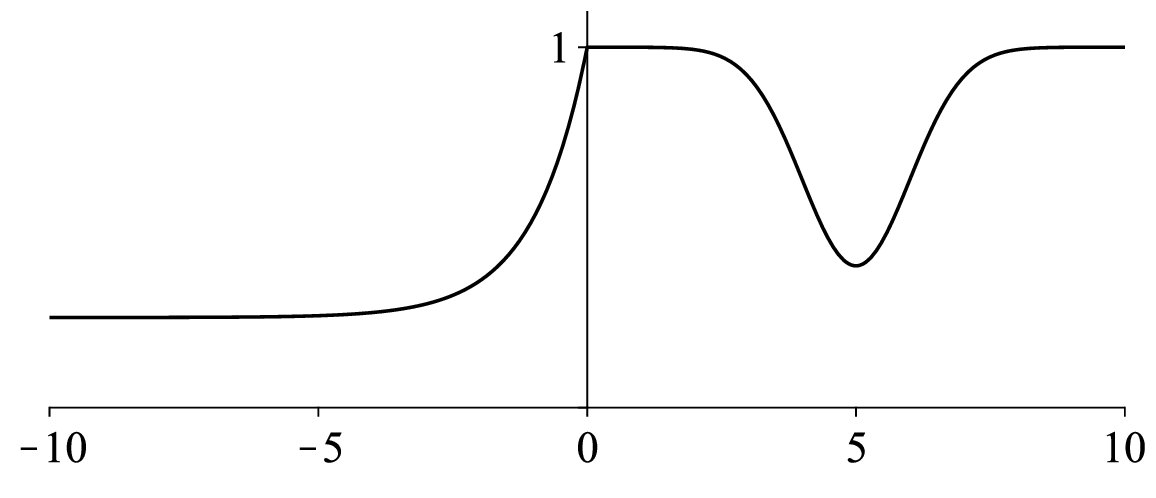}\\
\includegraphics[width=0.4\textwidth]{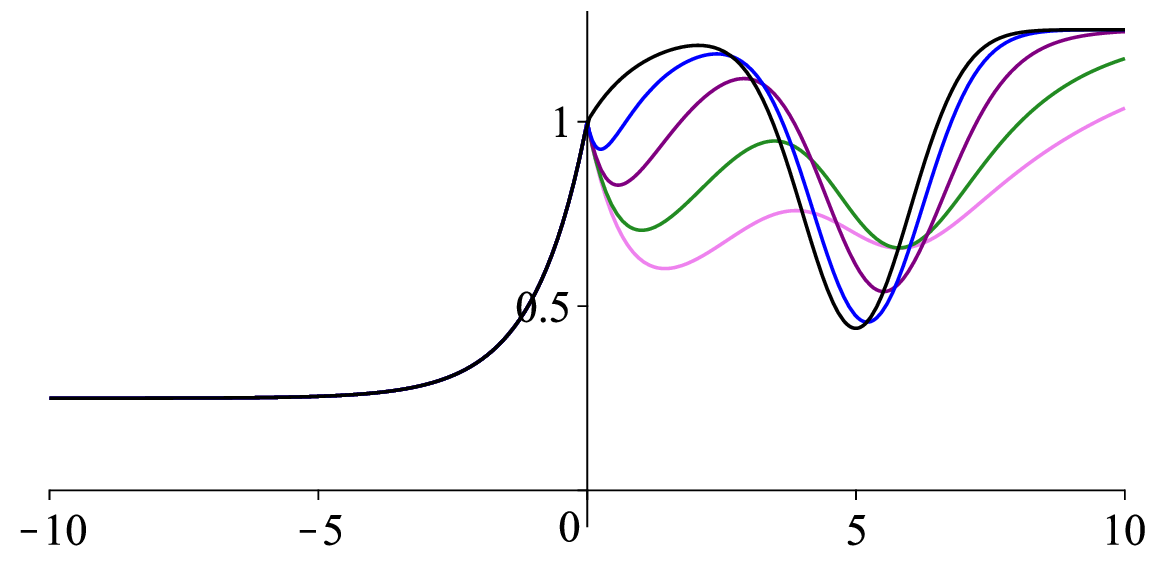}
\includegraphics[width=0.4\textwidth]{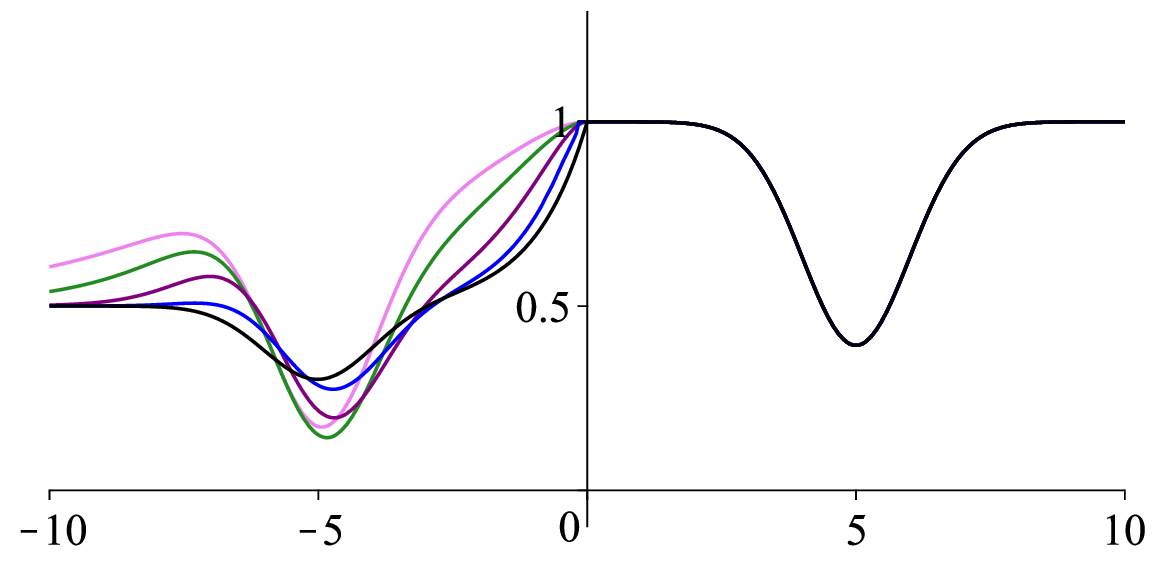}
\caption{Top: example of $f \in \cer$. Bottom:
$ E_{\alpha,\beta}^{\textnormal{skew}} f$ (black) and $\eson f $ for $n=1,2,5,15$ (pink, green, purple, blue, resp.) and $\alpha=0.2, \beta=0.1$. Graphs of $\fl$'s are on the left and of $\fr$'s on the right. }
\label{rys}
\end{figure}

\subsection{Proof of Theorem \ref{thm:skew}}
\hspace{0.6cm} \emph{Proof of (a)} In view of \eqref{ab:6}, estimate \eqref{skew:2} is a direct consequence of \eqref{ab:5a} because the operator norms of $R$ and $\ced(t)$ are equal to $1.$ 

\emph{Proof of (b)} By \eqref{ab:6} and \eqref{ab:6a}, to establish \eqref{skew:3}, it suffices to prove that  
\begin{equation}\label{skew:5} \gran \eson f = E_{\alpha,\beta}^{\textnormal{skew}} f, \qquad f \in \cer .\end{equation}
This, in turn, is an immediate consequence of the fact that distributions of exponential random variables with large parameters converge to Dirac measure at $0$, as expressed in Lemma \ref{lem:dirnew} presented in Appendix. For, Lemma \ref{lem:dirnew} (c)  implies that for $f\in \cer$, the second term in the second line  of the 
definition \eqref{needed} of $\widetilde {f_\ell} $ 
 with  $\alpha$ and $\beta$ replaced by $n\alpha$ and $n\beta$, respectively,  converges, as $n\to \infty$, to $\frac{\alpha}{\alpha+\beta} (f(x) - f(-x))$ uniformly in $x\ge 0.$ 
 This is because the function $\phi$ defined by $\phi (x) = f(x) - f(-x), x\ge 0$ belongs to $\cerp$, and we have $\phi (0)=0$. It follows that the first coordinate of $ \eson f $  converges (in the norm of $\cer$) to $\widetilde {f_{\ell}}$ defined by  \eqref{skew:4.1}---see Fig. \ref{rys}.
Since, similarly, its second coordinate converges to $ \widetilde {f_{\textnormal r}}$ from \eqref{skew:4.2}, the proof of \eqref{skew:5} is completed, and so is the proof of \eqref{skew:3}.

Before completing the proof of (b), we recall the general theorem shown in \cite{zwojtkiem} (see also \cite{knigaz}*{Chapter 61})  saying that outside of the regularity space (i.e., outside of the subspace defined in \eqref{skew:reg}---the regularity space of a sequence of cosine families is by definition the regularity space of the sequence of corresponding semigroups) cosine families cannot converge.
Since in the case of cosine families $\{\cosson (t),\, t \in \R\}, n \ge 1$ the regularity space equals $\cer$, 
outside of $\cer$ the limit $\gran \cosson (t)f $ does not exist  for at least one  $t\in \R \setminus \{0\}$. 

However, the abstract Kelvin formula \eqref{ab:6} provides additional insight into convergence under study. Namely, 
for $f$ belonging to $\x$, but not to $\cer$,  $ \gran \eson f$ still exits and is given by \eqref{skew:4.1} and \eqref{skew:4.2}, but the convergence is no longer uniform in $x$ and in the limit we obtain a pair of functions which are continuous everywhere except for $x=0$---see Fig. \ref{rys2}. Hence, if $t\neq0$,  the pointwise $\gran\cosson (t)f $ is not an element of $\x$, being discontinuous at $x=\pm t$. In particular, for $f\not \in \cer$, the strong limit exists for no $t\not =0$. 

\begin{figure} [tb]\centering
\includegraphics[width=0.4\textwidth]{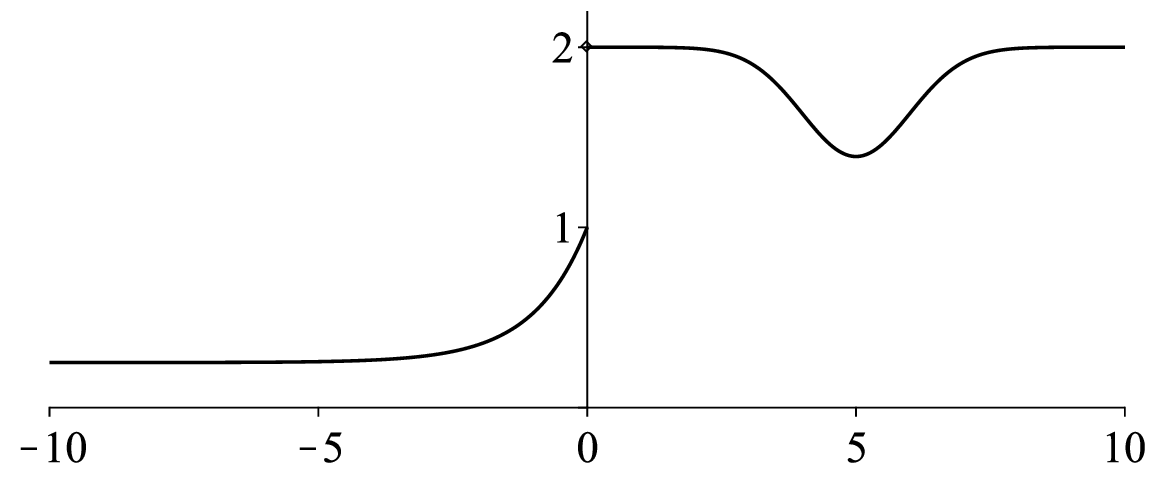}\\
\includegraphics[width=0.4\textwidth]{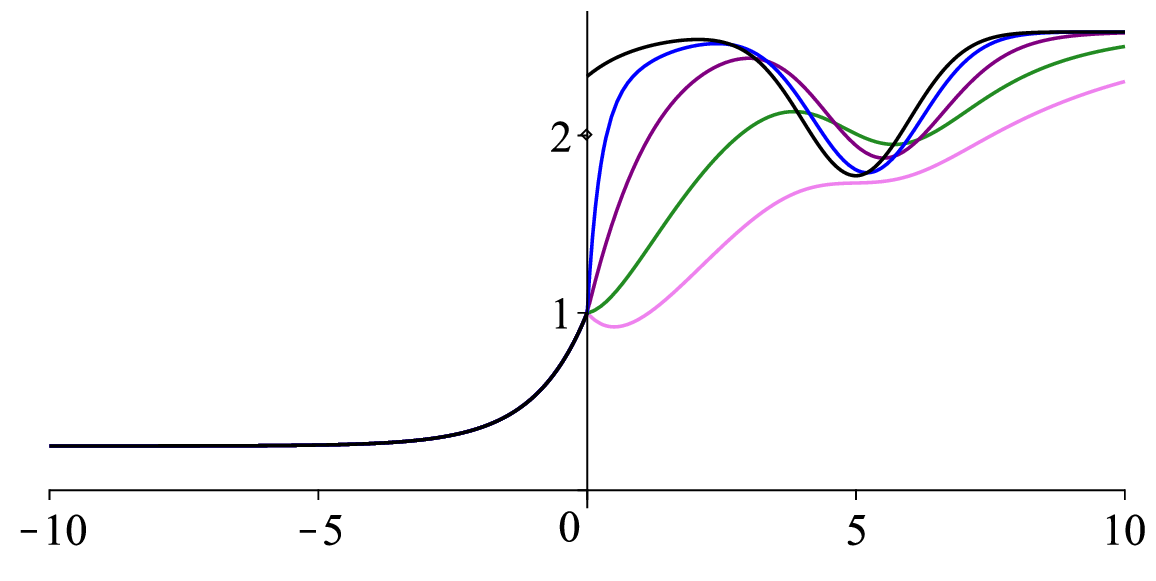}
\includegraphics[width=0.4\textwidth]{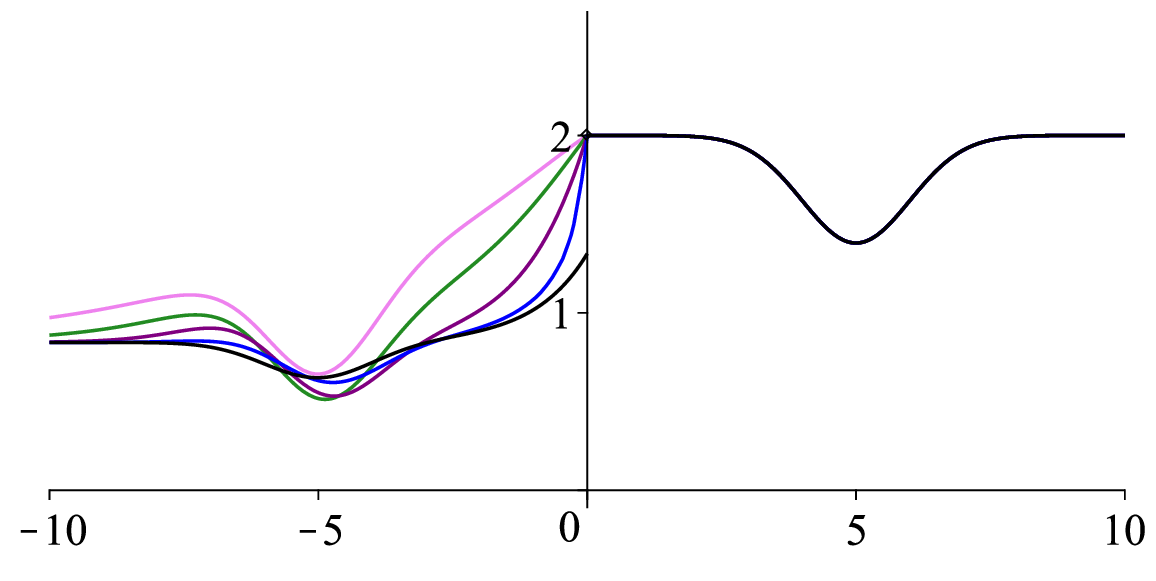}
\caption{Top: example of $f \in \x$. Bottom:
$\eson f $, $n=1,2,5,15$ (pink, green, purple, blue, resp.) and $\gran \eson f $ (black) for $\alpha=0.2, \beta=0.1$. Graphs of $\fl$'s are on the left and of $\fr$'s on the right.}
\label{rys2}
\end{figure}

\emph{Proof of (c)} 
We recall that, as an application of his algebraic version of the Hille--Yosida Theorem, J. Kisy\'nski has proved the following result accompanying Trotter--Kato--Neveu Theorem (see Corollary 5.8 in \cite{prof10}): if $B_n, n \ge 1$ are generators of equibounded semigroups, then the existence of the strong limit 
\( \gran \rez{B_n}  \)
is equivalent to the existence of 
\[ \gran \int_0^\infty \Psi (t) \e^{tB_n} \ud t \]
for any $\Psi$ that is absolutely integrable on $\R^+$ (see also \cite{hn}; the fact that the existence of the first limit implies the existence of the second for $\Psi = 1_{[0,\tau]}$ with $\tau>0$, has been noted already in the 1970 paper of T. G. Kurtz \cite{kurtz2}). 

An analogue of this result for cosine families was found in \cite{odessa}.  In our case it says that the existence of the strong limit $\gran \rez{\ason}$ (which is established in \cite{tombatty}) implies the existence of  the strong limit 
\begin{equation}\label{skew:6} \gran \int_{-\infty}^\infty \Psi (s) \cosson (s) \ud s \end{equation}
for any $\Psi$ that is absolutely integrable on the entire $\R$ and even (as long as we have an estimate of the type \eqref{skew:2}). 
Since, by the Weierstrass formula \eqref{cosf:1},
\[ \e^{t\ason} f =  {\textstyle \frac 1{2\sqrt{\pi t}}} \int_{-\infty}^\infty \e^{-\frac {s^2}{4t}} \cosson (s) f \ud s, \qquad f \in \x,  t >0\] 
the limit in \eqref{skew:4} is a particular case of \eqref{skew:6} for $\Psi(s)= \Psi_t (s) \coloneqq   {\textstyle \frac 1{2\sqrt{\pi t}}} \e^{-\frac {s^2}{4t}},$ $ s\in \R, t >0$. The fact that the limit is uniform with respect to $t$ in compact subintervals of $(0,\infty)$ is a consequence of the general result discussed in \cite{note} and \cite{knigaz}*{Thm 28.4}.

\begin{rem}\label{rem:rate} The rate of convergence in Theorem \ref{thm:skew} (b) depends apparently on the modulus of continuity of  $f\in \cer$. We can be more specific about this rate if we restrict ourselves to a dense subset of $\cer$: for $\phi \in C^1[-\infty,\infty]$, or, more generally, for $\phi$ that are Lipschitz continuous, the difference between $\cosson (t)f$  and  $\cskew (t) f$   is of the order of  $n^{-1}$. Indeed, if $|f(x)-f(y)|\le L |x-y|, x,y\in \R$ then $\phi\in \cerp$ defined by $\phi (x) \coloneqq f(x) - f(-x), x \ge 0$ is Lipschitz continuous on $\R^+$ with the same Lipschitz constant as $f$, and the the second line  in  \eqref{needed}  
 with  $\alpha$ and $\beta$ replaced by $n\alpha$ and $n\beta$, respectively, differs from the second line in \eqref{skew:4.1} by
\[ \frac {2\alpha}{\al + \be}[ n(\al +\be) e_{n(\al +\be)} * \phi (x) - \phi(x) ] \] 
and this in absolute value does not exceed $\frac {2\alpha L}{n(\al + \be)^2}$ by Remark \ref{rem:arate}. 
Since a similar reasoning shows that the absolute value of the difference between the first line in \eqref{ajednak} with $\al$ and $\be$ replaced by $n\al$ and $n\be$, respectively, and the first line in \eqref {skew:4.2} does not exceed  $\frac {2\beta L}{n(\al + \be)^2}$, we conclude that 
\[\sup_{t\in\R} \|  \cosson (t)f - \cskew (t) f \|\le \|  \eson f - E_{\alpha,\beta}^{\textnormal{skew}}f \| \le {\textstyle \frac Kn}, \]
where $K\coloneqq  \frac {2 \max (\alpha,\beta) L}{(\al + \be)^2}$. This, in turn, as in Remark \ref{rem:one}, implies 
\[  \sup_{t\geq 0}\big\|\e^{t\ason} f - \e^{t\askew} f\big\| \le {\textstyle \frac Kn}. \] 
 \end{rem}
\section{Asymptotic behavior} \label{sec:ab}

It is well-known  (see \cite{zwojtkiem}*{Corollary 1} which extends \cite{abhn}*{Prop. 3.14.6}) that, if  $\{C(t), t \in \R\}$ is a strongly continuous family in a Banach space $\mathsf F$, then the limit 
\( \grat C(t)f \)
exist for an $f\in \mathsf F$ only in the trivial case, that is, iff $C(t)f=f, t \in \R$. On the other hand, for rather non-trivial cosine families it often happens that
 the mean
 \[ M f \coloneqq  \grat t^{-1}\int_0^t C(s)f \ud s \]
exists for all $f \in \mathsf F$. For example, for the basic cosine family this limit exists in the bounded pointwise topology and we have 
\begin{equation}\label{ab:1}  Mf = \pol (f(-\infty) + f(\infty)), \qquad f \in \cer.\end{equation}
(Here, and in the rest of this section, for simplicity of notations, we do not distinguish between a real number, like 
$ \pol (f(-\infty) + f(\infty))$, and a function that is everywhere equal to this number.)
In this section we show that the mean operator $M$ for $\cosso$ remains the same when $\alpha$ and $\beta$ are replaced by $n\alpha$ and $n\beta$ respectively, and in fact coincides with the mean operator for $\cskew$. As a result, in view of Theorem \ref{thm:skew} (b), for $f\in \cer$ we can write
\begin{align*} \grat t^{-1}\int_0^t  \gran \cosson (s) f \ud s  &= \grat t^{-1}\int_0^t \cskew (s) f \ud s \\ &= \gran \grat t^{-1}\int_0^t \cosson  (s) f \ud s . \end{align*}

\begin{prop}\label{prop:srednie} We have 
\begin{align*} \grat t^{-1}\int_0^t \cosso (s) f \ud s &= \mab f, \qquad f \in \x,\\
\grat t^{-1}\int_0^t \cskew (s) f \ud s &=  \mab f, \quad f \in \cer, \end{align*}
where
\[ \mab f \coloneqq
{\textstyle \frac \beta{\alpha + \beta}} f(-\infty) + {\textstyle \frac \alpha{\alpha + \beta}} f(\infty),  \qquad   f \in \x. \] \end{prop}

\begin{proof}Because of \eqref{ab:1}, for $\widetilde{\fl}$ of \eqref{skew:4.1}, we have 
\[   \grat t^{-1}\int_0^t C(s) \widetilde{\fl} \ud s =\pol (f(-\infty) + {\textstyle \frac{\beta-\alpha}{\alpha+\beta}} f(-\infty) + \textstyle{\frac{2\alpha}{\alpha+\beta}} f(\infty)) = \mab f.\]
and a similar calculation based on \eqref{skew:4.2} shows that $\grat t^{-1}\int_0^t C(s)\widetilde{\fr} \ud s$ $=
\mab f$ too. Hence, $\grat t^{-1}\int_0^t \ced (s) \eskew f \ud s = (\mab f,\mab f)\in \ce $,
and thus 
\[  \grat t^{-1}\int_0^t \cskew (s) f \ud s = R  (\mab f, \mab f) = \mab f,\]
showing the second relation to be proved. 

The proof of the first of these relations is rather similar. For $\widetilde{\fl}$ of \eqref{needed} 
\[   \grat t^{-1}\int_0^t C(s) \widetilde{\fl} \ud s = f(-\infty) + {\textstyle \frac \alpha{\alpha+\beta}} (f(\infty)-f(-\infty)) = \mab f ,\]
because 
\( \lim_{x\to \infty} e_{\alpha+\beta}*[f - f^\es](x) = {\textstyle \frac 1{\alpha+\beta}} (f(\infty) - f(-\infty))\).
 Since the calculation for $\widetilde{\fr}$ gives the same result, we argue as above to complete the proof. 
 \end{proof}

\section{Convergence of projections}\label{sec:cop}

The abstract Kelvin formula \eqref{ab:6} 
works well because the extension operator
$\eso$ 
is shaped by transmission conditions \eqref{intro:2} in a very specific manner. In particular, the space \[\y\subset \ce \] of extensions of all elements of $\x$ is invariant under the basic cartesian product cosine family $\{\ced(t),t\in \R\}$, 

It is one of the main findings of \cite{aberman} that $\y$ 
is \emph{complemented}. In fact, Section 3 in \cite{aberman} provides an explicit form of a natural projection of $\ce$ on this subspace. A  `nearly orthogonal' projection, say,  $\p$ on $\y$ is given by 
\[ \p (f_1,f_2) = (g_1,g_2) \]
where 
\begin{align}\label{intro:eqdef_g2}\begin{split}
g_1(x) & = \fle^e(x) +\pol(\gamma+2\al) \int_{-\infty}^x \big[\textstyle{\frac 1\gamma} k_1(y)-\pol  k_2(y) \big]\e^ {-\sq  (x-y)}   \ud y \\
 &\phantom{=} -\pol(\gamma-2\al)  \int_x^{\infty} \big[\textstyle{\frac 1\gamma} k_1(y)+\pol  k_2(y) \big]\e^ {\sq (x-y)}   \ud y ,\\
g_{2}(x) &= \fre^e(x) +\pol (\gamma+2\be) \int_x^\infty  \big[\textstyle{\frac 1 \gamma} k_1(y)+\pol  k_2(y) \big]\e^{\sq (x-y)}   \ud y \\   
 &\phantom{=} -\pol (\gamma-2\be ) \int_{-\infty}^x \big[\textstyle{\frac 1 \gamma} k_1(y)-\pol  k_2(y) \big]\e^{-\sq (x-y)}   \ud y , \qquad x\in \R,   \end{split}
 \end{align}
and 
\begin{align*}
 k_1 \coloneqq   \alpha f_1^o + \beta f_2^o, \qquad
 k_2\coloneqq f_1^e - f_2^e .  \end{align*}

We will show that convergence of solution families described in the previous section is accompanied by convergence of the related projections. We will show namely that projections $\pn$  converge strongly to a projection on the space of extensions related to the skew Brownian motion. 

\begin{thm}\label{conv_proj1} \ \ 
\begin{itemize}
\item [(a)]
For every $(f_1,f_2)\in \ce$, we have
\begin{align*}
\lim_{n\to \infty} \pn(f_1,f_2)=\pskew(f_1,f_2),
\end{align*}
where \[\pskew(f_1,f_2)\coloneqq \big(f_1^e+\textstyle{\frac{2\al}{\gamma^2}} k_1-\pol k_2, f_2^e+\textstyle{\frac{2\be}{\gamma^2}} k_1+\pol k_2\big).\]
\item [(b)]
The map $\pskew$ is a projection on the space 
\[ \yy \coloneqq \{ (g_1,g_2) \in \ce\colon g_1^e=g_2^e, \be g_1^o= \al g_2^o \}.\] 
\item [(c)] The space $\yy$ is precisely the subspace of extensions related to the skew Brownian motion  generated by  $\askew$.
\end{itemize}
\end{thm}

\begin{proof}
\textbf{ (a)} From the first formula in \eqref{intro:eqdef_g2}
 it follows that the first coordinate of $\pn(f_1,f_2)$ is
\begin{align*}
g_{1,n}(x) & = \fle^e(x) +\pol(\gamma+2\al)n \int_{-\infty}^x \big[\textstyle{\frac 1\gamma} k_1(y)-\pol  k_2(y) \big]\e^ {-n\sq  (x-y)}   \ud y \nonumber \\
 &\phantom{=} -\pol(\gamma-2\al) n \int_x^{\infty} \big[\textstyle{\frac 1\gamma} k_1(y)+\pol  k_2(y) \big]\e^ {n\sq (x-y)}   \ud y , \qquad x\in \R.  \end{align*}
Lemma \ref{lem:dirnew} in Appendix implies now that $g_{1,n}$ converges in the norm of $\cer$, as $n\to \infty$, to 
$
f_1^e+ \frac{2\al}{\gamma^2} k_1-\frac{1}{2} k_2$, as desired.

For the proof of convergence of the second coordinate of $\pn(f_1,f_2)$, denoted  $g_{2,n}$,  we proceed analogously. Namely, by \eqref{intro:eqdef_g2}, we have
\begin{align*}
g_{2,n}(x) &= \fre^e(x) +\pol (\gamma+2\be)n \int_x^\infty  \big[\textstyle{\frac 1 \gamma} k_1(y)+\pol  k_2(y) \big]\e^{n \sq (x-y)}   \ud y \\
 &\phantom{=} -\pol (\gamma-2\be )n \int_{-\infty}^x \big[\textstyle{\frac 1 \gamma} k_1(y)-\pol  k_2(y) \big]\e^{-n\sq (x-y)}   \ud y , \qquad x\in \R.   
 \end{align*}
Hence, by Lemma \ref{lem:dirnew}, $g_{2,n}$ converges in the norm of $\cer$ to 
$f_2^e+\frac{2\be}{\gamma^2}k_1+\frac{1}{2}k_2$.

\textbf{ (b)} {As a strong limit of projections, $\pskew$ is a projection also. To characterize its range we} consider $(f_1,f_2) \in \ce$ and $(g_1,g_2)= \pskew(f_1,f_2) $.
Since $f_1^e,f_2^e,$ $ k_1$ and $ k_2$ all belong to $\cer$, it follows that so do $g_1$ and $g_2 $. Furthermore,
$g_1^e=f_1^e-\pol  k_2=\pol(f_1^e+f_2^e)=f_2^e+\pol  k_2=g_2^e$, because $k_1$ is odd and $k_2$ is even. 
Similarly, we see that $g_1^o=\frac{2\al}{\gamma^2}  k_1$ and $g_2^o=\frac{2\be}{\gamma^2}  k_1$. 
This shows  that $(\gl,\gr) \in \yy$, that is, that the range of $\pskew$ is contained in $\yy$. 

To prove the other inclusion, we assume that  $\be f_1^o= \al f_2^o$ and $f_1^e = f_2^e.$
Then, $2\alpha  k_1 = \gamma^2 f_1^o$, $2\beta  k_1 = \gamma^2 f_2^o$ and $ k_2=0$. Hence the definition of $\gl$ and $\gr$  simplifies to $\gl =f_1^e+f_1^o=f_1$ and $\gr = f_2^e+f_2^o=f_2$. This means that for $(f_1,f_2) \in \yy$, $\pskew (f_1,f_2)=(f_1,f_2)$. Therefore, the range of $\pskew$ contains $\yy$ and the proof of (b) is complete. 

\textbf{ (c)} Relations \eqref{skew:4.1} and \eqref{skew:4.2} show that  the subspace  of extensions related to the skew Brownian motion can be characterized as the space of  $(\gl,\gr)\in\ce$ satisfying $\gl(x)= \frac{\be-\al}{\apb} \gl(-x)+\frac{2\al}{\apb} \gr(x)$ and $ \gr(-x)=\frac{2\be}{\apb} \gl(-x)+\frac{\al-\be}{\apb} \gr(x) $ for $x\geq 0$.
Multiplying by $\alpha + \beta$ and rearranging (using \eqref{elek}) we obtain the following equivalent form of these equations:
\begin{align}
\be g_1^o&= \al( g_2-g_1^e),\qquad \al g_2^o= \be( g_2^e-g_1^\es) \quad  \textrm{on } \R^+. \label{dodane:1}
\end{align}
These, in turn, are equivalent to 
\begin{align}
g_1^e&= g_2^e,\qquad \be g_1^o= \al g_2^o \quad  \textrm{on } \R^+. \label{dodane:2}
\end{align}
Indeed, summing sides of \eqref{dodane:1} and rearranging once again yields the first relation in \eqref{dodane:2}, and then the first relation in \eqref{dodane:1} renders the second in \eqref{dodane:2}. Conversely, a straightforward argument leads from \eqref{dodane:2} to \eqref{dodane:1}. Finally, equations featured in \eqref{dodane:2} hold on $\R^+$ whenever they hold on $\R$.  
Hence, the space of extensions equals $\yy$. \end{proof}

\begin{rem}
The operator $\pskew$ mimics orthogonal projection in the Hilbert space of square integrable functions on $\R$. More precisely, for $(f_1,f_2)\in \ce$, the pair $\pskew(f_1,f_2)$ minimizes the value of
\[
L(g_1,g_2)\coloneqq \int_{-y}^y (g_1(x)-f_1(x))^2+(g_2(x)-f_2(x))^2 \ud x
\]
over all $(g_1,g_2)\in \yy$ for any positive $y$.
To prove this claim, we assume for a moment that $\alpha \neq \beta$. Summing the sides of the first equation in \eqref{dodane:2} multiplied by $\alpha$ with the sides of the second, we check that  $\gl(-x)= \frac{\apb}{\be-\al} \gl(x)-\frac{2\al}{\be-\al} \gr(x)$ for $(g_1,g_2) \in \yy$ and $x\ge 0$. Also, repeating these calculations with $\alpha$ replaced by $\beta$ yields $ \gr(-x)=\frac{2\be}{\be-\al} \gl(x)-\frac{\apb}{\be-\al} \gr(x) $ for $x\geq 0$.  It follows that
 \begin{align*}
L(g_1,g_2)=& \int_{0}^y (g_1(x)-f_1(x))^2+\Big({\textstyle \frac{\apb}{\be-\al}} \gl(x)-{\textstyle \frac{2\al}{\be-\al}} \gr(x) -f_1(-x) \Big)^2
\\&+(g_2(x)-f_2(x))^2 +\Big({\textstyle \frac{2\be}{\be-\al}} \gl(x)-{\textstyle \frac{\apb}{\be-\al}} \gr(x) -f_2(-x)   \Big)^2\ud x.
\end{align*}
We can minimize the integrand  for each $x \in (0,y)$ by finding the minimum of a quadratic function in two variables ($f_1(\pm x)$ and $f_2(\pm x)$ are treated as fixed parameters). Direct calculations verify that such solution coincides with $\pskew(f_1,f_2)$ on the interval $(-y,y)$.

The case $\al=\be$ is straightforward because it comes down to the fact that $(g_1,g_2)=(\pol(f_1+f_2),\pol(f_1+f_2))=P^{\textnormal{skew}}_{\al,\al}(f_1,f_2)$ miminizes
 $L(g_1,g_2)=\int_{-y}^y (g_1(x)-f_1(x))^2+(g_1(x)-f_2(x))^2 \ud x$ over all $(g_1,g_2)\in \ce$ satisfying $g_1=g_2$.
\end{rem}

\begin{rem} Arguing as in Remark \ref{rem:rate}, we can estimate the rate of convergence in Theorem \ref{conv_proj1} (a); we omit the details. \end{rem}

\section{Convergence of `complementary' families}\label{sec:cocf}

\subsection{Definition of `complementary' cosine families}\label{sec:docc}
The fact that $\p$ of Section \ref{sec:cop} is a projection on $\y\subset \ce$ implies the existence of the subspace $\z$ that complements $\y$ to the entire $\ce$:
\[ \ce = \y \oplus \z ;\] 
$\z$ is defined as the kernel of $\p$ or, equivalently, as the image of $\ce$ via \begin{equation} \q\coloneqq I_\ce - \p.\label{defq} \end{equation}
Interestingly, $ \z$ turns out to be invariant under the cartesian product basic cosine family also, and in fact is unequivocally shaped by certain boundary conditions. 

Elaborating on this succinct statement, we first define the space \[ \xo \] as the subspace of functions $f\in \x$ satisfying $f(0+)=-f(0-)$ (`$\textnormal{ov}$' stands for `opposite values'). As it transpires, each member $f$ of $\xo$ can be identified with the pair \[ \esop f = (\wfl,\wfr)\in \z\]
determined by
\begin{align*}
 \wfl(x)&= \begin{cases*}f(x),& $x<  0$, \\ 
 -f(-x)-2f(0+)e_\su(x)-2 e_\su * [\be f-\al f^\es](x), & $x> 0,$  \end{cases*}
 \end{align*}
and
\begin{align*}
  \wfr (x)&= \begin{cases*}-f(-x)+2f(0+)e_\su(-x)+ 2 e_\su * [\be f-\al f^\es](-x) , & $x< 0$,\\
f(x),&  $x>0$.  \end{cases*} \end{align*}  
As a result, the formula 
\begin{equation} \label{gener:B} \cossot (t)  \coloneqq   R C_D(t) \esop, \qquad t \in \R,\end{equation}
defines a cosine family in $\xo$. Since \(\|\esop\|\le 5$, we have 
\[ \|\cossot (t)\| \le 5, \qquad \alpha, \beta \ge 0, t \in \R. \]
It is also proved in \cite{aberman} that the generator, say, $A^{\textnormal{o-s}}_{\al,\be}$, of this cosine family can be characterized as follows. 
Its domain $D(A^{\textnormal{o-s}}_{\al,\be})$ consists of functions $f \in \xo$ satisfying the following
conditions: 
\begin{itemize}
\item[(a)] $f$ is twice continuously differentiable in both
  $(-\infty,0-]$ and $[0+,\infty)$, separately, with left-hand and
  right-hand derivatives at $x=0\mp$, respectively, 
\item[(b)] both the limits $\lim_{x\to \infty} f'' (x)$ and
  $\lim_{x\to -\infty} f''(x) $ exist and are finite, and
 \item[(c)]\mbox{}\vspace{-1.5\baselineskip} \begin{equation}\label{komplementarne} f''(0+)=\alpha f'(0-) + \beta f'(0+) \mquad{ and } f''(0+)=-f''(0-).\end{equation}    \end{itemize} 
 Moreover, for such $f$ we have
 \[ A^{\textnormal{o-s}}_{\al,\be} f = f'';\]
we note in passing that the second condition in \eqref{komplementarne} is necessary for  $f''$ to belong to $ \xo$. 

It is worth stressing two points here. First of all, because of the way $\cossot$ was defined, it can be seen as a `complementary' cosine family to $\cosso$. Moreover, it can be shown that if \eqref{gener:B} is to define the cosine family generated by $A^{\textnormal{o-s}}_{\al,\be}$, the extension operator $\esop$ has to be defined as above. In this sense, $\esop$, and thus $\z$ also, is shaped by the boundary conditions \eqref{komplementarne}. 

\subsection{Convergence of `complementary' solution families}

We know from Theorem  \ref{thm:skew} that cosine families $\cosson $ converge to the cosine families $\cskew$ as $n\to \infty$, and that implies convergence of the related semigroups. In this section, we want to complement the aforementioned theorem with the result concerning convergence of  $\{\cossot,\, t\in \R\}$ introduced in Section \ref{sec:docc}. We will show that (see Theorem \ref{thm:weks}), in contrast to $\cosson$ which converge only on a subspace of $\x$, $\cossotn$ converge, as $n\to \infty$, on the entire $\xo$. Moreover, perhaps more surprisingly,  the limit cosine family is a mirror image of that related to skew Brownian motion (see Proposition \ref{prop:skew=weks}).

\begin{thm}\label{thm:weks} We have  
\begin{equation}\label{weks:3} \gran \sup_{t \in \R} \| \cossotn (t)f - \cweks (t)f\| =0, \qquad  f \in \xo, \end{equation}
where $\cweks (t)\coloneqq R \ced (t)  E_{\al,\be}^\textnormal{weks}$ and  $E_{\alpha,\beta}^\textnormal{weks} f = (\widetilde {f_{\ell}},\widetilde {f_{\textnormal r}})$ is defined by
\begin{align}\label{weks:1.1}
 \widetilde {f_\ell}(x)&= \begin{cases*}f(x),& $x<  0$, \\ 
\frac{\al-\be}{\apb} f(-x)-\frac{2\be}{\apb} f(x), & $x> 0,$  \end{cases*}
 \end{align}
and
\begin{align}\label{weks:1.2}
  \widetilde {f_{\textnormal r}}(x)&= \begin{cases*} -\frac{2\al}{\apb} f(x)+\frac{\be-\al}{\apb} f(-x) , &  $x< 0$,\\
f(x),& $x>0$.  \end{cases*} \end{align}  
As a consequence, 
\begin{equation*} \gran \sup_{t \ge 0}\big \| \e^{tA^{\textnormal{o-s}}_{n\al,n\be}}  - T(t) f\big\| =0, \qquad  f \in \xo, \end{equation*}
where $T(0)f =f $ and 
\begin{equation*} T(t) f  = {\textstyle \frac 1{2\sqrt{\pi t}}} \int_{-\infty}^\infty \e^{-\frac {s^2}{4t}} \cweks (s) f  \ud s, \qquad t >0, f \in \xo. \end{equation*} 
\end{thm}

\begin{proof} It suffices to show \eqref{weks:3}, because the rest then can be proved as in Remark \ref{rem:one}. Indeed, $A^{\textnormal{o-s}}_{n\al,n\be}$ is the generator of $\cossotn$, and $\cweks$ is a cosine family as a limit of cosine families. But, to establish \eqref{weks:3} it suffices to check that 
\begin{equation}\label{weks:5} \gran \esopn  f = E_{\alpha,\beta}^{\textnormal{weks}} f, \qquad f \in \xo \end{equation}
for $\esop$ defined in the preceding section.

To this end, we note that, for $f\in \xo$,  the continuous function $\phi$ on $\R^+$ determined by $\phi (x) = \be f(x) - \al f(-x), x> 0$, (in particular $\phi (0) =(\alpha+\beta)f(0+)$) belongs to $\cerp$. Therefore, by
Lemma \ref{lem:dirnew} (c), 
\[ f(0+)e_{n(\alpha+\beta)}(x)+n e_{n(\su)} * [\be f-\al f^\es](x) \]
 converges, as $n\to \infty$, to $(\alpha+\beta)^{-1}[\be f(x)-\al f^\es(x)]$ uniformly in $x\ge 0.$
 It follows that the first coordinate of $ \esopn f $ converges in the norm of $\cer$ to $\wfl$ defined by  \eqref{weks:1.1}.
Since, similarly, the second coordinate  of $ \esopn f $ converges to $ \widetilde {f_{\textnormal r}}$ of \eqref{weks:1.2}, the proof is complete. 
 \end{proof}


The space $\xo$ in which $\cweks$ is defined is somewhat unusual, and it is perhaps the main reason why $\cweks$  remains a bit mysterious even if a direct characterization of its generator is provided (that characterization can be given, but does not seem to be very informative).  Fortunately, $\xo$ is isometrically isomorphic to the space $\cer$ where a number of  semigroups and cosine families can be interpreted probabilistically, and it turns out that the image of $\cweks$ in $\cer$ coincides with $\cskewr$---that is, describes the skew Brownian motion with the roles of $\alpha$ and $\beta$ exchanged.  

To make these remarks more precise,  we introduce the isometric isomorphism 
\[ J\colon\xo \to \cer \]
of   $\xo$ and $\cer$, given by $J f (x)= -f(x), x<0$, $Jf(x)=f(x), x >0$ and $Jf(0)= f(0+)$. In terms of $J$ our result takes the following form.  
\begin{prop}\label{prop:skew=weks} For all $t\in \R$, 
\[  J \cweks (t)J^{-1}=\cskewr (t). \]

\end{prop}
\newcommand{\eweks}{E_{\al,\be}^\textnormal{weks}}
\begin{proof}
Let $\mc J$ be the isometric automorphism of  $\ce $ given by $\mc J (f_1,f_2) = (-f_1,f_2)$. It is clear that  $\mathcal{J}$ is its own inverse. We note the following three identities that can be established by a straightforward calculation. 
\begin{itemize} 
\item [(i) ] $ \eweks J^{-1}= \mc J E_{\be,\al}^\textnormal{skew} $.
\item [(ii) ] $\mc J\ced (t) \mc J = \ced (t),t\in \R$.
\item [(iii) ]$ JR(f_1,f_2) =R \mc J (f_1,f_2),$ provided that $f_1(0)=-f_2(0)$. 
\end{itemize}
Hence, 
\begin{align*} J \cweks (t)J^{-1} &= (J R) \ced (t) (E_{\alpha,\beta}^{\textnormal{weks}} J^{-1}) = (R\mc J) \ced (t) (\mc J E_{\be,\al}^\textnormal{skew})\\ &= R\ced (t) E_{\be,\al}^\textnormal{skew}=\cskewr (t), \end{align*}
as desired. 
\end{proof}

We have established that, as $n\to \infty$, along with $\cosson$ converge also  complementary families $\cossotn$; the first of these families converge to $\cskew$, the second to a mirror image of $\cskewr$.

\begin{rem} In Theorem \ref{thm:weks} (a) the rate of convergence can also be estimated. Arguing as in Remark \ref{rem:rate} we can prove that, provided that $Jf$ is Lipschitz continuous with Lipschitz constant $L$, $\sup_{t \in \R} \| \cossotn (t)f - \cweks (t)f\| \le \frac Kn$ for $K\coloneqq \frac{2L}{\alpha+\beta} $. \end{rem}

\section{Additional results}\label{sec:ar}
To complete the paper, we need to answer two natural questions. First of all, as a direct consequence of Theorem \ref{conv_proj1} (a), 
\[ \gran \qn (f_1,f_2) = \qweks (f_1,f_2) \]
where $\q$ was defined in \eqref{defq} whereas 
\[ \qweks (f_1,f_2) \coloneqq \big(f_1^o -\textstyle{\frac{2\al}{\gamma^2}} k_1+\pol k_2, f_2^o-\textstyle{\frac{2\be}{\gamma^2}} k_1-\pol k_2\big).\]
It is natural to ask whether $\qweks$ has anything to do with $\cweks$. As the following result shows, the answer is in the affirmative: $\qweks$ projects on the invariant space of extensions related to $\cweks$, that is, on the range of the operator $\eweks$.

\begin{prop} \label{conv_proj2}
The map $\qweks$ is a projection on the space
\[  \zzz \coloneqq \{ (g_1,g_2) \in \ce\colon g_1^e=-g_2^e, \al g_1^o= -\be g_2^o\}.\] 
Moreover, $\zzz$ is the range of $\eweks$.  
\end{prop}
\begin{proof} We start by noting that $\qweks$ is a projection, being the limit of projections. Hence, to prove the first statement it suffices to find the range of $\qweks$. To this end, for $(f_1,f_2) \in \ce$, let $(g_1,g_2)= \qweks(f_1,f_2) $, that is,
$g_1=f_1^o-\textstyle{\frac{2\al}{\gamma^2}} k_1+\pol k_2$ and $g_2= f_2^o-\textstyle{\frac{2\be}{\gamma^2}} k_1-\pol k_2$.
The fact that $g_1,g_2\in \cer$ follows from point (b) in Theorem \ref{conv_proj1}. Also,
as in  the proof of that result, we obtain
$\gl^e=\pol(f_1^e-f_2^e)=-\gr^e$ and $\al g_1^o=\frac{2\al\be}{\gamma^2}(\be f_1 ^o-\al f_2^o)= -\be g_2^o$ 
implying that $(\gl,\gr) \in \zzz.$ 
Moreover, it is easy to check that, if  $(f_1,f_2)\in \zzz $, then $\gl=f_1$ and $\gr= f_2$. This proves the first statement in the proposition. 

Identity (i) presented in the proof of Theorem \ref{prop:skew=weks} can be rewritten as  $ \eweks = \mc J E_{\be,\al}^\textnormal{skew} J$. Since the range of $J$ is the entire $\cer$, the range of $\eweks$ is the image of the range of  $ E_{\be,\al}^\textnormal{skew} $ under $\mc J$. However, the range of   $ E_{\be,\al}^\textnormal{skew} $ is characterized in Theorem \ref{conv_proj1} (b). This characterization reveals that the image of the range of   $ E_{\be,\al}^\textnormal{skew} $ is precisely $\zzz$ as defined above. \end{proof}

\begin{cor}\label{tcztery} The space $\ce $ is a direct sum of two spaces that are invariant under the basic Cartesian product cosine family: 
\begin{align*} \ce = \yy \oplus \zzz .\end{align*} \end{cor}

The second, and final, question one need to answer in this paper reads: is there a counterpart of Proposition \ref{prop:srednie} pertaining to the complementary cosine families $\cossot$? The answer is given below; we omit the proof of this statement, since it is very similar to that of Proposition \ref{prop:srednie}.
\newcommand{\nab}{N_{\alpha,\beta}}

\begin{prop}\label{prop:srednie2} For all $f\in \xo$, 
\begin{align*} \grat t^{-1}\int_0^t \cossot (s) f \ud s =
\grat t^{-1}\int_0^t \cweks (s) f \ud s &=  \nab f,  \end{align*}
where $\nab f \in \xo$ is defined as follows
\[ \nab f (x) \coloneqq \begin{cases}
{\textstyle \frac \alpha{\alpha + \beta}} f(-\infty) - {\textstyle \frac \beta{\alpha + \beta}} f(\infty),& x<0, \\
{\textstyle \frac \beta{\alpha + \beta}} f(\infty)  - {\textstyle \frac \alpha{\alpha + \beta}} f(-\infty), & x>0. \end{cases}     \]  \end{prop}
We note that this result agrees nicely with Proposition \ref{prop:srednie} and Proposition \ref{prop:skew=weks}, for we have 
\begin{align*} \grat t^{-1}\int_0^t  \cskewr (s) f \ud s &=
J \grat t^{-1}\int_0^t  \cweks (t)J^{-1} f \ud s\\& =  J \nab J^{-1}f = M_{\beta,\alpha} f. \end{align*}

\section{Appendix}

The following lemma is indispensable in proving our convergence theorems.
\begin{lem}\label{lem:dirnew}Let $a>0$, $\phi \in\cer$ and $\varphi \in \cerp$ be given. Then, 
\begin{itemize}
\item [(a)] $\displaystyle \lim_{n \to \infty} n a \int_{-\infty}^x \e^{-n a (x-y)}\phi (y) \ud y = \phi (x), $ \\
\item [(b)] $ \displaystyle \lim_{n \to \infty}  n a \int_x^{\infty} \e^{n a (x-y)}\phi (y) \ud y = \phi (x) 
$, \item [(c)] \( \displaystyle  \lim_{n \to \infty} \left [ na \int_0^x \e^{-na (x-y)}\varphi (y) \ud y + \e^{-na x}\varphi (0)\right ]= \varphi (x). \)
\end{itemize}
The first two limits are uniform with respect to $x\in \R$, the third is uniform with respect to $x\ge 0.$
\end{lem}
\begin{proof} To prove (a),  we apply a typical argument involving so-called Dirac sequences (see e.g. \cite{lang}*{pp. 227--235} or \cite{feller}*{pp. 219--220}). Namely, given $\eps>0$ we can find a $\delta$ such that $|\phi(x)-\phi(y)|< \frac \eps 2$ as long as $|x-y|< \delta$; this is because members of $\cer$ are uniformly continuous. Having chosen such a $\delta$ we can also find an $n_0$ so large that $2 \e^{-na \delta}\|\phi\| < \frac \eps 2$, provided that $n>n_0$, where $\|\phi\|\coloneqq \|\phi\|_{\cer}$.

Next, for any $x \in \R$ we have 
\[ na \int_{-\infty}^x \e^{-na (x-y)} \phi (y) \ud y - \phi (x) =  na \int_0^{\infty} \e^{-na y} [\phi (x-y) -\phi (x)] \ud y .\]
It follows that for $n>n_0$, the absolute value of the left-hand side above does not exceed 
\[ na \int_0^\delta \e^{-na y} {\textstyle \frac \eps 2} \ud y +   na \int_\delta^\infty \e^{-na y} 2\|\phi\| \ud y  = 
{\textstyle \frac \eps 2}(1 - \e^{-na \delta})   + 2 \e^{-na \delta}\| \phi\| < \eps .\]
This completes the proof, $\eps >0$ being arbitrary. 

Condition (b) is  (a) in disguise: (a) becomes (b) if $\phi$ and $x$ are replaced by $\phi^\es$ and $-x$, respectively. Likewise, to obtain (c) we use (a) for $\phi\in \cer$ defined as follows: $\phi (x) = \varphi (x), x \ge 0$ and $\phi (x) = \varphi (0), x <0.$ 
\end{proof}

\begin{rem}\label{rem:arate} The proof presented above provides no rate of convergence in (a)--(c), but suggests that this rate depends on modulus of continuity of $\phi $ and $\varphi$.  If we restrict ourselves to $\phi$ that are Lipschitz continuous, (a) and (b) can be expressed in a more qualitative form. To wit, if $|\phi (x) - \phi (y)|\le L|x-y|$ for all $x,y\in \R$ and certain $L>0$, we have
\begin{align*} \left |n a \int_{-\infty}^x \e^{-n a (x-y)}\phi (y) \ud y  - \phi (x) \right |& = \left | na \int_0^\infty \e^{-na z} [\phi (x-z) - \phi (x)]\ud z \right |\\
&\le L na \int_0^\infty \e^{-na z} z \ud z = \frac L{na}. \end{align*}
Similarly, $| n a \int_x^{\infty} \e^{n a (x-y)}\phi (y) \ud y -\phi (x) | \le \frac L{na}$. Moreover, if $|\varphi (x)-\varphi (y)|\le L$ for $x,y\ge 0$, then $|  na \int_0^x \e^{-na (x-y)}\varphi (y) \ud y + \e^{-na x}\varphi (0) - \varphi (x)| \le \frac L{na}$.
  \end{rem}


\bibliographystyle{plain}
\bibliography{bibliografia}

\end{document}